\documentclass[12 pt]{amsart}
\usepackage{amsmath, amsthm, amssymb}    
\usepackage{graphicx}			
\usepackage[colorlinks=true, urlcolor=blue, citecolor=black, linkcolor=black]{hyperref}  
\usepackage{subcaption}			
\usepackage[table,xcdraw]{xcolor}

\newcommand{\Z}{\mathbb Z}

\theoremstyle{plain} 
\newtheorem{theorem}{Theorem}
 
\newtheorem{corollary}[theorem]{Corollary} 
 
\newtheorem{lemma}[theorem]{Lemma}
\theoremstyle{definition}

\author{David Richeson}
\title{How Much String to String a Cardioid?}
\date{\today}                                           
\begin{document}
\thispagestyle{empty}
\maketitle
\begin{abstract}
A residue design is an artistic geometric construction in which we have $n$ equally-spaced points on a circle numbered 0 through $n-1$, and we join with a line segment each point $k$ to $ak$ modulo $n$ for some fixed $a\ge 2.$ The envelopes of these lines are epicycloids, like cardioids. In this note, we prove that the sum of the lengths of these line segments has a surprisingly simple closed form. In particular, if one wants to make one of these designs with string, it is easy to calculate how much string is required.
\end{abstract}

Place $n$ equally-spaced points around a circle, label them 0 through $n-1$, and for some choice of $a\ge 2$, draw a line segment from each point $k$ to the point $ak$ modulo $n$. The envelope of these lines is an epicycloid (see, e.g., \cite{Beardon:1989}). Figure~\ref{fig:chord5} shows examples when $n=83$---a cardioid ($a=2$), a nephroid ($a=3$), and a trefoiloid ($a=4$). Such drawings are known as \emph{residue designs}  (see, e.g., \cite{Akopyan:2015, Archibald:1900,Johnson:1998,Locke:1972,Moore:1981,Perucca:2024,Picard:1971,Richeson:2023}). 

\begin{figure}[ht]
\includegraphics{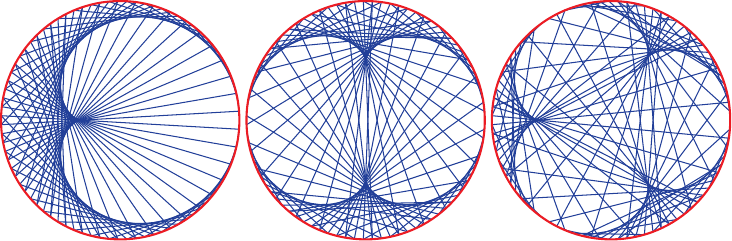}
\caption{Residue designs formed from 83 points on a circle and multiplicative factors of 2, 3, and 4: a cardioid, a nephroid, and a trefoiloid, respectively.}
\label{fig:chord5}
\end{figure}

It is possible to make these designs using pencil and paper or computer illustration software. This geometry also explains the caustics in the bottom of your coffee mug when light shines in from a point at its edge (cardioid) or from infinitely far away (nephroid). But perhaps the most aesthetically pleasing versions are made using string. We could place $n$ nails in a circle and run a string between nails $k$ and $ak$ modulo $n$ for all $k$. Made that way, the natural question is: How much string do we need? In this article, we prove that the exact length can be expressed in a surprisingly concise form.

Theorem~\ref{thm:stringlength2} answers our string-length question for many $a$ and $n$. 

\begin{theorem}
\label{thm:stringlength2}
In a circle of radius $r$, the sum of the distances from the point $k$ to $ak$ mod $n$ for $k=0,1,2,\ldots,n-1$ is \[S(n,a,r):=2rg\cot\left(\frac{\pi g}{2n}\right),\] where $g=\gcd(a-1,n)$.
\end{theorem}

However, for some $n$ and $a$, $S(n,a,r)$ may double-count some segments. For instance, when $n=56$ and $a=3$, $3\cdot 7= 21$ and $3\cdot 21=63\equiv 7\bmod 56.$ So, $S(56,3,r)$ counts the distance between 7 and 21 twice, whereas the residue design would require only one string between these points. In the language of graph theory, we don't want our graph to have parallel edges. Theorem~\ref{thm:segmentlength3} gives the general theorem---a closed form for the sum without double-counting these lengths.

\begin{theorem}
\label{thm:segmentlength3}
The sum of the lengths of the line segments in a residue design formed from $n$ points in a circle of radius $r$ and a multiplicative factor $a$ is \[L(n,a,r):=2rg_1\cot\left(\frac{\pi g_1}{2n}\right)-rg_2\cot\left(\frac{\pi g_2}{2m}\right),\] where $m$ is the largest divisor of $n$ for which $a^2\equiv 1\bmod m$, $g_1=\gcd(a-1,n)$, and $g_2=\gcd(a-1,m)$.
\end{theorem}

Of course, to make one of these designs with string, we need only the approximate string length. Thus, we have the following easier-to-compute formula.

\begin{corollary}\label{cor:stringlength}
For $r$, $n$, $a$, and $m$ as in Theorem~\ref{thm:segmentlength3}, $L(n,a,r)\approx (4n-2m)r/\pi.$ 
\end{corollary}

This note is organized as follows. We prove Theorem~\ref{thm:stringlength2} in Section~\ref{sec:proof} and Theorem~\ref{thm:segmentlength3} and Corollary~\ref{cor:stringlength} in Section~\ref{sec:doublecount}. In Section~\ref{sec:onestring}, we combine our results with the results in \cite{Richeson:2023} about residue designs that can be made with one continuous string to obtain Corollary~\ref{cor:prime}. 

\section{The Sum of the Distances.}
\label{sec:proof}
The proof of Theorem~\ref{thm:stringlength2} relies on a trigonometric identity due to Lagrange. We omit the proof, but to prove it, one can apply the identity $\sin\alpha\sin\beta=\tfrac12(\cos(\alpha-\beta)-\cos(\alpha+\beta))$ to $\sin(j\theta)=\sin (j\theta)\sin(\tfrac\theta 2)$ to obtain a telescoping sum, or one can use complex numbers and geometric series.

\begin{lemma}[Lagrange's identity]
If $\theta$ is not a multiple of $2\pi$, then 
\[\sum_{j=0}^m \sin (j\theta)
=\frac{\sin \left(\frac{m\theta}{2}\right)\sin \left(\frac{(m+1)\theta}{2}\right)}{\sin\left(\frac\theta 2\right)}.
\]
\end{lemma}
We also need the following lemma, the proof of which is straightforward.

\begin{lemma}\label{lem:cordlength}
The length of a chord in a circle of radius $r$ corresponding to a central angle of $\theta$ is $2r\sin(\theta/2)$.
\end{lemma}

\begin{proof}[Proof of Theorem~\ref{thm:stringlength2}]
Suppose $n$ equally-spaced points on a circle of radius $r$ are numbered $0,1,2,\ldots,n-1$ and that for each $k$, we measure the length of the segment joining the point $k$ to $ak$ modulo $n$. (Some segments may be measured twice.) Each such segment connects a point to the point $(a-1)k$ modulo $n$ points away in the numbering. The full collection of such ``distances'' are $0,a-1, 2(a-1), 3(a-1),\ldots, (n-1)(a-1)$ modulo $n$. 

If we view $\{0,1,2,\ldots,n-1\}$ as the additive cyclic group $\Z_n$, then \[G:=\{0,a-1, 2(a-1), 3(a-1),\ldots, (n-1)(a-1)\}\] is the subgroup of $\Z_n$ generated by $a-1$. This subgroup contains  $m=n/g$ elements where $g=\gcd(a-1,n)$ is the smallest positive generator of $G$. Moreover, every element of $G$ appears $g$ times in the list $0,a-1, 2(a-1), 3(a-1),\ldots, (n-1)(a-1)$ modulo $n$.

By Lemma~\ref{lem:cordlength}, the length of the chord joining two points $j$ points apart in our residue design is $2r\sin\left(\frac{\pi j}{n}\right).$ By Lagrange's identity,
\begin{align*}
S(n,a,r)&=g\sum_{i=0}^{m-1}2r\sin\left(\frac{\pi ig}{n}\right)\\
&=\frac{2rg\sin \left((m-1)\frac{\pi g}{2n}\right)\sin \left(m\cdot\frac{\pi g}{2n}\right)}{\sin\left(\frac{\pi g}{2n}\right)}\\
&=\frac{2rg\sin \left(\left(\frac{n}{g}-1\right)\frac{\pi g}{2n}\right)\sin \left(\frac{\pi}{2}\right)}{\sin\left(\frac{\pi g}{2n}\right)}\\
&=\frac{2rg\sin \left(\frac{\pi}{2}-\frac{\pi g}{2n}\right)}{\sin\left(\frac{\pi g}{2n}\right)}\\
&=\frac{2rg\cos \left(\frac{\pi g}{2n}\right)}{\sin\left(\frac{\pi g}{2n}\right)}\\
&=2rg\cot\left(\tfrac{\pi g}{2n}\right).
\end{align*}

\end{proof}

An interesting consequence of the proof of Theorem~\ref{thm:stringlength2} is that when $a-1$  and $n$ are relatively prime, such as when $n$ is a prime number, the line segments in the residue design are the same lengths as the ones leaving 0 and connecting to the other $n-1$ points on the circle, as shown in Figure~\ref{fig:chord4}. 

\begin{figure}[ht]
\includegraphics{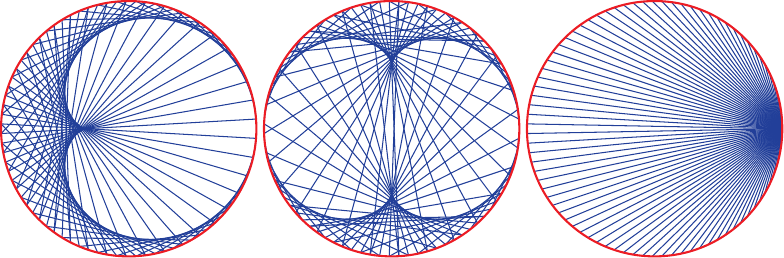}
\caption{When $a-1$ and $n$ are relatively prime, the lengths of the segments of a residue design are the same as those emanating from a single point.}
\label{fig:chord4}
\end{figure}

\section{Doubled Line Segments.}
\label{sec:doublecount}
Theorem~\ref{thm:stringlength2} gives the sum of the lengths of the line segments in a residue design. But it double-counts those segments that happen to run from $s$ to $t$ and $t$ to $s$. In this section, we investigate how to adjust the sum to account for these doubled segments.

Suppose we are given $n$ and $a$. Our focus is on integers $s$ and $t$ such that $at\equiv s\bmod n$ and $as\equiv t\bmod n$. If $t\not\equiv s\bmod n$, these correspond to segments in the design that are measured twice. If $t\equiv s\bmod n$, then the line segment is degenerate---it has no length. However, since $0+0=0$, we can, at least from an arithmetic point of view, consider this degenerate segment as doubled. Thus, the doubled or degenerate segments correspond precisely to the set \[H=H_{n,a}:=\{s\in\Z_n: a^2s\equiv s\bmod n\}.\]

We omit the straightforward proof that $H$ is a subgroup of $\Z_n$, and because $\Z_n$ is cyclic, so is $H$. If $k$ is a generator of $H$, we will use the familiar notation $H=\langle k\rangle.$

\begin{lemma}\label{lem:cyclicsubgroup}
If $n,a>0$, then $H_{n,a}$ is a cyclic subgroup of $\Z_n$.
\end{lemma}

There are many $n$ and $a$ for which there are no doubled segments. In such a case, either $H=\{0\}$ or $H$ consists only of $s$-values such that $as\equiv s$, and which contribute segments of length 0. When $n$ is prime and $a\not\equiv \pm 1\bmod n$, $H=\{0\}$ and no lines are doubled. But $n$ need not be prime for this to occur; for instance, when $n=40$ and $a=2$, $H=\{0\}$. When $n=40$ and $a=6$, $H=\langle8\rangle$, but since the line segments for 0, 8, 16, 24, and 32 are all degenerate, there are still no doubled segments.  

At the other extreme, when $a= n-1\equiv -1\bmod n$, $H=\Z_n$. In this case, every line segment is doubled except the degenerate ones at 0 and $n/2$ (if $n$ is even). The following lemma characterizes instances where every line segment is doubled or degenerate.

\begin{lemma}\label{lem:squareequivto1}
If $n,a>0$, then $H_{n,a}=\Z_n$ if and only if $a^2\equiv 1\bmod n$.
\end{lemma}
\begin{proof}
Let $n,a>0$. Suppose $H_{n,a}=\Z_n$. Then $a^2s\equiv s\bmod n$ for all $s\in \Z_n$. Because this is true for $s=1$, $a^2\equiv 1\bmod n$. Conversely, suppose $a^2\equiv 1\bmod n$. Then for all $s\in\Z_n$, $a^2s\equiv s\bmod n$. Hence $H=\Z_n.$ 
\end{proof}

Typically, for a given $n$, we have a mixture of behaviors for different $a$-values. When $n=46$, for instance, only two $a$-values produce doubled segments. When $a=45$, every segment is doubled, and when $a=22$, $H=\langle 2\rangle$, so every segment leaving an even number is doubled. However, when $n=40$, there are 20 $a$-values with nondegenerate doubled segments. When $a=3,5,7,13,15,23,27,35,37$,  $H=\langle 5\rangle$; when $a=4,14,24,34$, $H=\langle 8\rangle$; and because $a^2\equiv 1\bmod 40$ when $a=9,11,19,21,29,31,39$, $H=\Z_{40},$ and so every segment is doubled or degenerate.

The following lemma allows us to state precisely which segments are doubled or degenerate for a given $n$ and $a$.

\begin{lemma}\label{lem:subgroupgenerator}
If $n,a>0$, then $H_{n,a}=\langle \frac{n}{m}\rangle$ where $m$ is the largest divisor of $n$ for which $a^2\equiv 1\bmod m$.
\end{lemma}
\begin{proof}
Let $n$ and $a$ be positive integers. By Lemma~\ref{lem:cyclicsubgroup}, $H=H_{n,a}$ is a cyclic subgroup of $\Z_n$ of order $m$, say. So, $H=\langle j\rangle$ where $j=\frac{n}{m}.$ Now, let's look only at the segments in the residue design corresponding to the values in $H=\{0,j,2j,\ldots,(m-1)j\}.$ It is geometrically identical to the residue design with $m$ points, $\{0,1,2,\ldots,m-1\},$ also with a multiplicative factor of $a$. In this residue design, every segment is either doubled or degenerate. Thus, $H_{m,a}=\Z_m.$ By Lemma~\ref{lem:squareequivto1}, $a^2\equiv 1\bmod m$. Lastly, for any divisor $m'$ of $n$ such that $a^2\equiv 1\bmod m',$ $\langle\frac{n}{m'}\rangle\subseteq H$, and the larger $m'$, the larger the set $\langle\frac{n}{m'}\rangle$. We obtain $H$ from the largest such divisor, $m.$
\end{proof}

To illustrate Lemma~\ref{lem:subgroupgenerator} and its proof, consider $n=56$ and $a=3$. The divisors of 56 are 1, 2, 4, 7, 8, 14, 28, and 56. When we compute $a^2=9$ modulo these values, we obtain 1 only for 1, 2, 4, and 8. So, $m=8$, and $H=\langle\frac{56}{8}\rangle=\langle 7\rangle$. Indeed, the segments leaving 7, 14, 21, 35, 42, and 49 are all doubled, and those leaving 0 and 28 have length zero. As shown in Figure~\ref{fig:doubled}, the doubled and degenerate segments are precisely the same as the full residue design with $m=8$ points and $a=3.$

\begin{figure}[ht]
\includegraphics{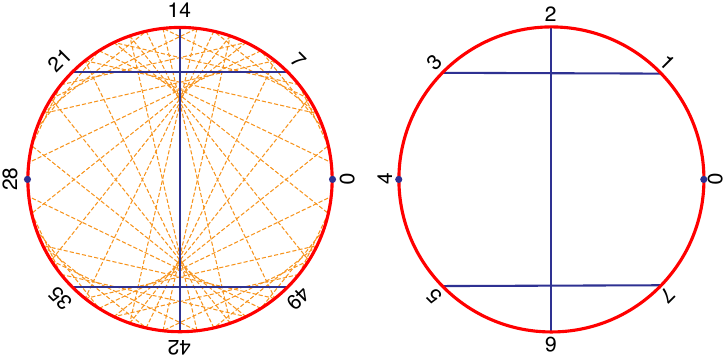}
\caption{The doubled or degenerate segments (the solid lines and the marked points on the circle) for $n=56$ and $a=3$ are the same as the full residue design for $n=8$ and $a=3$.}
\label{fig:doubled}
\end{figure}

\begin{proof}[Proof of Theorem~\ref{thm:segmentlength3}] To prove this theorem, we must simply subtract off those segments that are double-counted by the the formula in Theorem~\ref{thm:stringlength2}. By Lemma~\ref{lem:subgroupgenerator}, \[L(n,a,r)=S(n,a,r)-\tfrac12S(m,a,r).\]
\end{proof}

\begin{proof}[Proof of Corollary~\ref{cor:stringlength}]
The series expansion for the cotangent function (\cite[p.~123]{Aigner:2001}) is\[\cot x=\frac{1}{x}+\sum_{i=1}^{\infty}\frac{(-1)^i2^{2i}B_{2i}}{(2i)!}x^{2i-1}=\frac1x-\frac{x}{3}-\frac{x^3}{45}\cdots\] for $0<x<\pi$, where $B_n$ is the $n$th Bernoulli number. Moreover, because $g<\sqrt{n},$  \[S(n,a,r)=2rg\cot\left(\frac{\pi g}{2n}\right)=\frac{4 n r}{\pi}-\frac{\pi r g^2}{3n}-\frac{\pi^3g^4r}{180n^3}-\cdots\approx \frac{4 n r}{\pi}.\]
By Theorem~\ref{thm:segmentlength3}, the string length required is
\[L(n,a,r)=S(n,a,r)-\tfrac12 S(m,a,r)\approx \frac{4 n r}{\pi}-\frac{2 m r}{\pi}=\frac{(4 n-2m) r}{\pi}.\]
\end{proof}

Returning again to our example with $n=56$ and $a=3$, shown in Figure~\ref{fig:doubled}, and assuming the circle has radius $r=5$~cm, we find $m=8$, $g_1=\gcd(2,56)=2$, and $g_2=\gcd(2,8)=2$. So, the total length of line segments is 
\[L(56,3,5)=S(56,3,5)-\frac{S(8,3,5)}{2}=20\cot\left(\frac{\pi}{56}\right)-10\cot\left(\frac{\pi}{8}\right)\approx 332\text{ cm}.\]
Using the approximation in Corollary~\ref{cor:stringlength} we find that \[L(56,3,5)\approx (4\cdot 56-2\cdot 8)5/\pi\approx 331\text{ cm}.\]

\section{Residue Designs with One String.}
\label{sec:onestring}
Residue designs with $n$ points have $n$ line segments (some of which, like the segment from 0 to 0, may have zero length), but we would not want to use $n$ pieces of string to make the design. It would be better to find an $n$-value so we can run the string sequentially from nail to nail to nail with a small number of strings---preferably only one.

For instance, the cardioid design in Figure~\ref{fig:chord5} requires only one string. It runs from nail 1 to 2 to 4 to 8, and so on (the progression continues: 16, 32, 64, 45, 7, 14, 28, 56, 29, 58, 33, 66, 49, 15, 30, 60, 37, 74, 65, 47, 11, 22, 44, 5, 10, 20, 40, 80, 77, 71, 59, 35, 70, 57, 31, 62, 41, 82, 81, 79, 75, 67, 51, 19, 38, 76, 69, 55, 27, 54, 25, 50, 17, 34, 68, 53, 23, 46, 9, 18, 36, 72, 61, 39, 78, 73, 63, 43, 3, 6, 12, 24, 48, 13, 26, 52, 21, 42, 1). The nephroid in Figure~\ref{fig:chord5} can also be made with one string, whereas the trefoiloid requires two.

The article \cite{Richeson:2023} addresses how many strings a residue design requires. In particular, a residue design with $n$ points and multiplicative factor $a$ can be strung with one string if and only if $n$ is prime and $a$ is a primitive root modulo $n$; that is, $a$ is a generator of the cyclic group $\Z_n^\times.$ For $a=2$, these primes are 3, 5, 11, 13, 19, 29, 37, 53, 59, 61, 67, 83, and so on. For $a=3$, the primes are 5, 7, 17, 19, 29, 31, 43, 53, 79, 83, and so on. When $a=4$, there are no such primes.

Given the results in \cite{Richeson:2023}, the following corollary follows immediately from Theorem~\ref{thm:stringlength2} and Corollary~\ref{cor:stringlength}.

\begin{corollary}\label{cor:prime}
Let $a$ be a primitive root modulo a prime number $p$. One string of length $2r\cot(\tfrac{\pi}{2p})\approx 4pr/\pi$ suffices to make a string-art residue design with $p$ nails in a circle of radius $r$ and using a multiplicative factor $a$.
\end{corollary}

This shows, for instance, that to make the cardioid or nephroid in Figure~\ref{fig:chord5} in a circle of radius of 5~cm, we would need one string of length $2\cdot 5\cdot \cot(\frac{\pi}{2\cdot83})\approx 528$~cm.

    
{\setlength{\baselineskip}{13pt} 
\raggedright				
\bibliographystyle{plain}
\bibliography{HowMuchString}
} 

\end{document}